\documentclass[reqno]{amsproc}
\usepackage{amssymb}
\usepackage{amsfonts}

\theoremstyle{plain}

\newtheorem{definition}{Definition}

\newtheorem{lemma}{Lemma}

\newtheorem{remark}{Remark}

\newtheorem{theorem}{Theorem}
\numberwithin{equation}{section}
\input{tcilatex}

\begin{document}
\title{A Dunkl Analogue of Operators Including Two-variable Hermite
polynomials}
\author{Rabia Akta\c{s}}
\address{Ankara University, Faculty of Science, Department of Mathematics,
06100, Tando\u{g}an, Ankara, Turkey}
\email{raktas@science.ankara.edu.tr}
\author{Bayram \c{C}ekim}
\curraddr{Gazi University, Faculty of Science, Department of Mathematics,
06100, Be\c{s}evler, Ankara, Turkey}
\email{bayramcekim@gazi.edu.tr}
\author{Fatma Ta\c{s}delen}
\address{Ankara University, Faculty of Science, Department of Mathematics,
06100, Tando\u{g}an, Ankara, Turkey}
\email{tasdelen@science.ankara.edu.tr}
\subjclass[2000]{Primary 41A25, 41A36; Secondary 33C45}
\keywords{Dunkl analogue, Hermite polynomial, modulus of continuity,
Korovkin's type approximation theorem}

\begin{abstract}
The aim of this paper is to introduce a Dunkl generalization of the
operators including two variable Hermite polynomials which are defined by
Krech \cite{K} (Krech, G. A note on some positive linear operators
associated with the Hermite polynomials, Carpathian J. Math., 32 (1) \textbf{%
(}2016\textbf{)}, 71--77) and to investigate approximating properties for
these operators by means of the classical modulus of continuity, second
modulus of continuity and Peetre's K-functional.
\end{abstract}

\maketitle

\section{Introduction}

Up to now, linear positive operators and their approximation properties have
been studied by many research workers, see for example \cite{Atakut}, \cite%
{Ata-Buyuk}, \cite{Ciupa}, \cite{Gadzhiev}, \cite{Gupta}, \cite{Stancu}, 
\cite{Szasz} and references therein.\textbf{\ }Also, linear positive
operators defined via generating functions and their further extensions are
intensively studied by a large number of authors. For various extensions and
further properties, we refer for example\ Altin et.al \cite{ADT}, Dogru et.
al \cite{DOT}, Olgun et.al \cite{OIT}, Sucu et. al \cite{Sucu et al.},
Tasdelen et.al \cite{TAA}, Varma et.al \cite{Varma et al., VT}.

Recently, linear positive operators generated by a Dunkl generalization of
the exponential function have been stated by many authors. In \cite{Sucu},
Dunkl analogue of Sz\'{a}sz operators by using Dunkl analogue of exponential
function was given as follows%
\begin{equation}
S_{n}^{\ast }\left( g;x\right) =\frac{1}{e_{\nu }\left( nx\right) }%
\sum_{k=0}^{\infty }\frac{\left( nx\right) ^{k}}{\gamma _{\nu }\left(
k\right) }g\left( \frac{k+2\nu \theta _{k}}{n}\right) \ ;\ n\in 
\mathbb{N}
,~\nu ,\ x\in \lbrack 0,\infty )\   \label{1}
\end{equation}%
for $g\in C[0,\infty )$ where Dunkl analogue of exponential function is
defined by%
\begin{equation}
e_{\nu }\left( x\right) =\sum_{k=0}^{\infty }\frac{x^{k}}{\gamma _{\nu
}\left( k\right) }  \label{7}
\end{equation}%
for $k\in 
\mathbb{N}
_{0}$ and $\nu >-\frac{1}{2}$ and the coefficients $\gamma _{\nu }$ are as
follows%
\begin{equation}
\gamma _{\nu }\left( 2k\right) =\frac{2^{2k}k!\Gamma \left( k+\nu
+1/2\right) }{\Gamma \left( \nu +1/2\right) }\text{ and }\gamma _{\nu
}\left( 2k+1\right) =\frac{2^{2k+1}k!\Gamma \left( k+\nu +3/2\right) }{%
\Gamma \left( \nu +1/2\right) }  \label{3}
\end{equation}%
in \cite{Rosenblum}. Also, the coefficients $\gamma _{\nu }$ verify the
recursion relation%
\begin{equation}
\frac{\gamma _{\nu }\left( k+1\right) }{\gamma _{\nu }\left( k\right) }%
=\left( 2\nu \theta _{k+1}+k+1\right) ,\text{ }k\in 
\mathbb{N}
_{0},  \label{4}
\end{equation}%
where%
\begin{equation}
\theta _{k}=\left\{ 
\begin{array}{cc}
0, & if\text{ }k=2p \\ 
1, & if\text{ }k=2p+1%
\end{array}%
\right.  \label{5}
\end{equation}%
for $p\in 
\mathbb{N}
\cup \left\{ 0\right\} .$ Similarly, Stancu-type generalization of Dunkl
analogue of Sz\'{a}sz-Kantorovich operators and Dunkl generalization of Sz%
\'{a}sz operators via q-calculus have been defined in \cite{IC, icoz} and
for other research see \cite{M1, M}.

The two-variable Hermite Kampe de Feriet polynomials $H_{n}(\xi ,\alpha )$
are defined by (see \cite{Appell})%
\begin{equation*}
\dsum\limits_{n=0}^{\infty }\frac{H_{n}(\xi ,\alpha )}{n!}t^{n}=e^{\xi
t+\alpha t^{2}}
\end{equation*}%
from which, it follows%
\begin{equation*}
H_{n}(\xi ,\alpha )=n!\dsum\limits_{k=0}^{\left[ \frac{n}{2}\right] }\frac{%
\alpha ^{k}\xi ^{n-2k}}{k!(n-2k)!}.
\end{equation*}

In a recent paper, Krech \cite{K} has introduced the class of operators $%
G_{n}^{\alpha }$ given by%
\begin{equation}
G_{n}^{\alpha }\left( f;x\right) =e^{-\left( nx+\alpha x^{2}\right) }\dsum
\limits_{k=0}^{\infty }\frac{x^{k}}{k!}H_{k}(n,\alpha )f\left( \frac{k}{n}%
\right) ~~,~~x\in \mathbb{R}_{0}^{+}~,~n\in \mathbb{N~}\text{~,~~}\alpha
\geq 0  \label{aa10}
\end{equation}%
in terms of two variable Hermite polynomials and investigated approximation
properties of $G_{n}^{\alpha }$ .

In the present paper, we first give the Dunkl generalization of two variable
Hermite polynomials and then we define a class of operators by using the
Dunkl generalization of two variable Hermite polynomials. We give the rates
of convergence of the operators $T_{n}$ to $f$ by means of the classical
modulus of continuity, second modulus of continuity and Peetre's $K$%
-functional and in terms of the elements of the Lipschitz class $%
Lip_{M}\left( \alpha \right) .$

\section{The Dunkl generalization of two variable Hermite polynomials}

The Dunkl generalization of two variable Hermite polynomials is defined by%
\begin{equation}
\dsum\limits_{n=0}^{\infty }\frac{H_{n}^{\mu }(\xi ,\alpha )}{n!}%
t^{n}=e^{\alpha t^{2}}e_{\mu }(\xi t)  \label{aa}
\end{equation}%
from which, we conclude%
\begin{equation*}
H_{n}^{\mu }(\xi ,\alpha )=n!\dsum\limits_{k=0}^{\left[ \frac{n}{2}\right] }%
\frac{\alpha ^{k}\xi ^{n-2k}}{k!\gamma _{\mu }(n-2k)},
\end{equation*}%
which gives the two variable Hermite polynomials as $\mu =0.\ $For our
purpose, we denote%
\begin{equation*}
h_{n}^{\mu }(\xi ,\alpha )=\dfrac{\gamma _{\mu }(n)H_{n}^{\mu }(\xi ,\alpha )%
}{n!}
\end{equation*}
and we can write that the polynomials $h_{n}^{\mu }(\xi ,\alpha )$ are
generated by%
\begin{equation}
\dsum\limits_{n=0}^{\infty }\frac{h_{n}^{\mu }(\xi ,\alpha )}{\gamma _{\mu
}(n)}t^{n}=e^{\alpha t^{2}}e_{\mu }(\xi t)  \label{a1}
\end{equation}%
where%
\begin{equation*}
h_{n}^{\mu }(\xi ,\alpha )=\gamma _{\mu }(n)\dsum\limits_{k=0}^{\left[ \frac{%
n}{2}\right] }\frac{\alpha ^{k}\xi ^{n-2k}}{k!\gamma _{\mu }(n-2k)}.
\end{equation*}%
In order to obtain some properties of $h_{n}^{\mu }(\xi ,\alpha ),$ we
remind the following definition and lemma given in \cite{Rosenblum}.

\begin{definition}
\cite{Rosenblum} Let $%
{\mu}%
\in \mathbb{C}_{0}~(\mathbb{C}_{0}:=\mathbb{C\setminus }\left \{ -\frac{1}{2}%
,-\frac{3}{2},...\right \} ,~x\in \mathbb{C}$ and let $\varphi $ be entire
function. The linear operator $\mathbb{D}_{\mu }$ is defined on all entire
functions $\varphi $ on $\mathbb{C}$ by%
\begin{equation}
\mathbb{D}_{\mu }(\varphi (x))=\varphi ^{^{\prime }}(x)+\frac{\mu }{x}%
(\varphi (x)-\varphi (-x)),\ x\in \mathbb{C}.  \label{6}
\end{equation}%
We use the notation $\mathbb{D}_{\mu ,x}$ since $\mathbb{D}_{\mu }$ is
acting on functions of the variable $x$. Thus, $\mathbb{D}_{\mu ,x}(\varphi
(x))=\left( \mathbb{D}_{\mu }\varphi \right) (x).$
\end{definition}

\begin{lemma}
\label{Lemma1} \cite{Rosenblum} Let $\varphi ,\psi $ be entire functions.
For the linear operator $\mathbb{D}_{\mu }$, the following statements hold%
\begin{equation*}
\begin{array}{cl}
i) & \mathbb{D}_{\mu }^{j}:x^{n}\rightarrow \frac{\gamma _{\mu }(n)}{\gamma
_{\mu }(n-j)}x^{n-j},j=0,1,2,...,n\ (n\in \mathbb{N)};~\mathbb{D}_{\mu
}^{j}:1\rightarrow 0 \\ 
&  \\ 
ii) & \mathbb{D}_{\mu }(\varphi \psi )=\mathbb{D}_{\mu }(\varphi )\psi
+\varphi \mathbb{D}_{\mu }(\psi ),\ \text{where }\varphi \text{ is an even
function} \\ 
&  \\ 
iii) & \mathbb{D}_{\mu }:e_{\mu }(\lambda x)\rightarrow \lambda e_{\mu
}(\lambda x).%
\end{array}%
\end{equation*}%
By using these definition and lemma, we can state the next result.
\end{lemma}

\begin{lemma}
\label{Lemma2}For the Dunkl generalization of two variable Hermite
polynomials $h_{n}^{\mu }(\xi ,\alpha )$, the following results hold true%
\begin{equation*}
\begin{array}{cl}
(i) & \dsum\limits_{n=0}^{\infty }\frac{h_{n+1}^{\mu }(\xi ,\alpha )}{\gamma
_{\mu }(n)}t^{n}=(\xi +2\alpha t)e^{\alpha t^{2}}e_{\mu }(\xi t) \\ 
(ii) & \dsum\limits_{n=0}^{\infty }\frac{h_{n+2}^{\mu }(\xi ,\alpha )}{%
\gamma _{\mu }(n)}t^{n}=(\xi ^{2}+4\xi \alpha t+4\alpha ^{2}t^{2}+2\alpha
)e^{\alpha t^{2}}e_{\mu }(\xi t)+4\alpha \mu e^{\alpha t^{2}}e_{\mu }(-\xi t)%
\end{array}%
\end{equation*}
\end{lemma}

\begin{proof}
Applying the linear operator $\mathbb{D}_{\mu }$ in view of Lemma \ref%
{Lemma1} $,$ we have%
\begin{equation}
\begin{array}{l}
\mathbb{D}_{\mu }(te_{\mu }(\xi t))=(t\xi +1)e_{\mu }(\xi t)+2\mu e_{\mu
}(-\xi t) \\ 
\mathbb{D}_{\mu }(e^{\alpha t^{2}})=2\alpha te^{\alpha t^{2}}.%
\end{array}
\label{100}
\end{equation}

Also applying the linear operator $\mathbb{D}_{\mu }$ to both side of the
generating function (\ref{a1}), we have%
\begin{equation*}
\dsum\limits_{n=0}^{\infty }\frac{h_{n}^{\mu }(\xi ,\alpha )}{\gamma _{\mu
}(n)}\mathbb{D}_{\mu }(t^{n})=\mathbb{D}_{\mu }(e^{\alpha t^{2}}e_{\mu }(\xi
t)).
\end{equation*}%
By using (\ref{100}) and Lemma \ref{Lemma1} (i), we get the first relation$.$
Similarly, if we apply the linear operator $\mathbb{D}_{\mu }$ to the
relation in (i), we get%
\begin{equation*}
\dsum\limits_{n=0}^{\infty }\frac{h_{n+1}^{\mu }(\xi ,\alpha )}{\gamma _{\mu
}(n)}\mathbb{D}_{\mu }(t^{n})=\mathbb{D}_{\mu }\left[ (\xi +2\alpha
t)e^{\alpha t^{2}}e_{\mu }(\xi t)\right]
\end{equation*}%
from (\ref{100}) and Lemma \ref{Lemma1}, it follows%
\begin{equation*}
\dsum\limits_{n=0}^{\infty }\frac{h_{n+2}^{\mu }(\xi ,\alpha )}{\gamma _{\mu
}(n)}t^{n}=(\xi ^{2}+4\xi \alpha t+4\alpha ^{2}t^{2}+2\alpha )e^{\alpha
t^{2}}e_{\mu }(\xi t)+4\alpha \mu e^{\alpha t^{2}}e_{\mu }(-\xi t).
\end{equation*}
\end{proof}

\begin{definition}
With the help of the Dunkl generalization of two variable Hermite
polynomials given in (\ref{a1}), we introduce the operators $T_{n}(f;x),$ $%
n\in 
\mathbb{N}
$ given by 
\begin{equation}
T_{n}(f;x):=\frac{1}{e^{\alpha x^{2}}e_{\mu }(nx)}\dsum\limits_{k=0}^{\infty
}\frac{h_{k}^{\mu }(n,\alpha )}{\gamma _{\mu }(k)}x^{k}f\left( \frac{k+2\mu
\theta _{k}}{n}\right)  \label{1001}
\end{equation}%
where $\alpha \geq 0,\mu \geq 0$ and $x\in \left[ 0,\infty \right) .$ The
operators (\ref{1001}) are linear and positive. In the case of $\mu =0,$ it
gives $G_{n}^{\alpha }$ given by (\ref{aa10})
\end{definition}

\begin{lemma}
\label{Lemma 3}For the operators $T_{n}(f;x),$ we can obtain the following
equations:%
\begin{equation*}
\begin{array}{cl}
(i) & T_{n}(1;x)=1 \\ 
(ii) & T_{n}(t;x)=x+\frac{2\alpha x^{2}}{n} \\ 
(iii) & T_{n}(t^{2};x)=x^{2}+\frac{4\alpha }{n^{2}}x^{2}+\frac{4\alpha }{n}%
x^{3}+\frac{4\alpha ^{2}}{n^{2}}x^{4}+\frac{x}{n}+\frac{2\mu x}{n}\frac{%
e_{\mu }(-nx)}{e_{\mu }(nx)}%
\end{array}%
\end{equation*}
\end{lemma}

\begin{proof}
By using the generating function in (\ref{a1}), the relation $(i)$ holds.
For the proof of $(ii),$ in view of the recursion relation in (\ref{4}), we
get%
\begin{equation*}
T_{n}(t;x)=\frac{1}{ne^{\alpha x^{2}}e_{\mu }(nx)}\dsum\limits_{k=1}^{\infty
}\frac{h_{k}^{\mu }(n,\alpha )}{\gamma _{\mu }(k-1)}x^{k}.
\end{equation*}%
When we replace $k$ by $k+1$, we obtain (ii) by use of Lemma \ref{Lemma2}
(i). For the proof of $(iii),$ by using (\ref{4})$,$ we have%
\begin{equation*}
T_{n}(t^{2};x)=\frac{x}{n^{2}e^{\alpha x^{2}}e_{\mu }(nx)}%
\dsum\limits_{k=0}^{\infty }(k+1+2\mu \theta _{k+1})\frac{h_{k+1}^{\mu
}(n,\alpha )}{\gamma _{\mu }(k)}x^{k}.
\end{equation*}%
From the equation%
\begin{equation}
\theta _{k+1}=\theta _{k}+(-1)^{k},  \label{x}
\end{equation}%
it yields%
\begin{eqnarray*}
T_{n}(t^{2};x) &=&\frac{x}{n^{2}e^{\alpha x^{2}}e_{\mu }(nx)}%
\dsum\limits_{k=0}^{\infty }(k+2\mu \theta _{k})\frac{h_{k+1}^{\mu
}(n,\alpha )}{\gamma _{\mu }(k)}x^{k} \\
&&+\frac{x}{n^{2}e^{\alpha x^{2}}e_{\mu }(nx)}\dsum\limits_{k=0}^{\infty
}(1+2\mu (-1)^{k})\frac{h_{k+1}^{\mu }(n,\alpha )}{\gamma _{\mu }(k)}x^{k}.
\end{eqnarray*}%
Using the recursion relation in (\ref{4}) in the first series, it follows%
\begin{eqnarray*}
T_{n}(t^{2};x) &=&\frac{x^{2}}{n^{2}e^{\alpha x^{2}}e_{\mu }(nx)}%
\dsum\limits_{k=0}^{\infty }\frac{h_{k+2}^{\mu }(n,\alpha )}{\gamma _{\mu
}(k)}x^{k}+\frac{x}{n^{2}e^{\alpha x^{2}}e_{\mu }(nx)}\dsum\limits_{k=0}^{%
\infty }\frac{h_{k+1}^{\mu }(n,\alpha )}{\gamma _{\mu }(k)}x^{k} \\
&& \\
&&+\frac{2\mu x}{n^{2}e^{\alpha x^{2}}e_{\mu }(nx)}\dsum\limits_{k=0}^{%
\infty }(-x)^{k}\frac{h_{k+1}^{\mu }(n,\alpha )}{\gamma _{\mu }(k)}
\end{eqnarray*}%
from Lemma \ref{Lemma2} (i) and (ii), we complete the proof of (iii).
\end{proof}

\begin{lemma}
\label{Lemma 4}As a consequence of Lemma \ref{Lemma 3}, we can give the next
results for $T_{n}$ operators%
\begin{eqnarray}
\Delta _{1} &=&T_{n}(t-x;x)=\frac{2\alpha x^{2}}{n}  \notag \\
\Delta _{2} &=&T_{n}(\left( t-x\right) ^{2};x)=\frac{1}{n^{2}}x\left(
4x^{3}\alpha ^{2}+4\alpha x+n\right) +\frac{2\mu x}{n}\frac{e_{\mu }(-nx)}{%
e_{\mu }(nx)}  \label{A}
\end{eqnarray}
\end{lemma}

\begin{theorem}
\label{Theorem 4}For $T_{n}$ operators and any uniformly continuous bounded
function $g$ on the interval $[0,\infty )$, we can give%
\begin{equation*}
T_{n}\left( g;x\right) \overset{\text{uniformly}}{\rightrightarrows }g\left(
x\right)
\end{equation*}
on each compact set $A\subset $ $[0,\infty )$\ when $n\rightarrow \infty $.

\begin{proof}
From Korovkin Theorem in \cite{Korovkin}, when $n\rightarrow \infty ,\ $we
have $T_{n}\left( g;x\right) \overset{\text{uniformly}}{\rightrightarrows }%
g\left( x\right) $ on $A\subset \lbrack 0,\infty )$ which is each compact
set because $\lim_{n\rightarrow \infty }T_{n}(e_{i};x)=x^{i},$ for$\
i=0,1,2, $ which is uniformly on $A\subset \lbrack 0,\infty )$ with the help
of using Lemma \ref{Lemma 4}.
\end{proof}
\end{theorem}

\begin{theorem}
The operator $T_{n}$ maps $C_{B}(%
\mathbb{R}
_{0}^{+})$ into $C_{B}(%
\mathbb{R}
_{0}^{+})$ $\ $and $\left \Vert T_{n}\left( f\right) \right \Vert \leq
\left
\Vert f\right \Vert $ for each $f\in C_{B}(%
\mathbb{R}
_{0}^{+})$ .
\end{theorem}

\section{Convergence of operators in (\protect\ref{1001})}

In what follows we give some rates of convergence of the operators $T_{n}$.
Firstly, we recall some definitions as follows. Let $Lip_{M}\left( \alpha
\right) $ Lipschitz class of order $\alpha .$ If $g\in Lip_{M}\left( \alpha
\right) $, the inequality%
\begin{equation*}
\left\vert g\left( s\right) -g\left( t\right) \right\vert \leq M\left\vert
s-t\right\vert ^{\alpha }
\end{equation*}%
holds where $s,t\in \lbrack 0,\infty ),\ 0<\alpha \leq 1$ and $M>0.$ $%
\widetilde{C}[0,\infty )$ is the space of uniformly continuous on $[0,\infty
).$ The modulus of continuity $g\in \widetilde{C}[0,\infty )$ is denoted by%
\begin{equation}
\omega \left( g;\delta \right) :=\sup\limits_{\substack{ s,t\in \lbrack
0,\infty )  \\ \left\vert s-t\right\vert \leq \delta }}\left\vert g\left(
s\right) -g\left( t\right) \right\vert .  \label{12}
\end{equation}%
We first estimate the rates of convergence of the operators $T_{n}$ by using
modulus of continuity and in terms of the elements of the Lipschitz class $%
Lip_{M}\left( \alpha \right) .$

\begin{theorem}
\label{Theorem 6}If $h\in Lip_{M}\left( \alpha \right) $, we have%
\begin{equation*}
\left\vert T_{n}\left( h;x\right) -h\left( x\right) \right\vert \leq M\left(
\Delta _{2}\right) ^{\alpha /2}
\end{equation*}
where $\Delta _{2}$ is given in Lemma \ref{Lemma 4}.
\end{theorem}

\begin{proof}
Since $h\in Lip_{M}\left( \alpha \right) $, it follows from linearity%
\begin{equation*}
\left\vert T_{n}\left( h;x\right) -h\left( x\right) \right\vert \leq
T_{n}\left( \left\vert h\left( t\right) -h\left( x\right) \right\vert
;x\right) \leq MT_{n}\left( \left\vert t-x\right\vert ^{\alpha };x\right) .
\end{equation*}%
From Lemma \ref{Lemma 4} and H\"{o}lder's famous inequality, we can write%
\begin{equation*}
\left\vert T_{n}\left( h;x\right) -h\left( x\right) \right\vert \leq M\left[
\Delta _{2}\right] ^{\frac{\alpha }{2}}.
\end{equation*}
Thus, we find the required inequality.
\end{proof}

\begin{theorem}
\label{Theorem 7}The operators in (\ref{1001}) verify the inequality%
\begin{equation*}
\left\vert T_{n}\left( g;x\right) -g\left( x\right) \right\vert \leq \left(
1+\sqrt{\frac{1}{n}x\left( 4x^{3}\alpha ^{2}+4x\alpha +n\right) +2\mu x\frac{%
e_{\mu }(-nx)}{e_{\mu }(nx)}}\right) \omega \left( g;\frac{1}{\sqrt{n}}%
\right) ,
\end{equation*}%
where $g\in \widetilde{C}[0,\infty ).$
\end{theorem}

\begin{proof}
By Lemma \ref{Lemma 4}, from Cauchy-Schwarz's inequality and the property of
modulus of continuity%
\begin{equation}
\left\vert g\left( t\right) -g\left( x\right) \right\vert \leq w\left(
g;\delta \right) \left( \frac{\left\vert t-x\right\vert }{\delta }+1\right) ,
\label{a8}
\end{equation}%
it follows%
\begin{eqnarray*}
\left\vert T_{n}\left( g;x\right) -g\left( x\right) \right\vert &\leq
&T_{n}\left( \left\vert g\left( t\right) -g\left( x\right) \right\vert
;x\right) \\
&\leq &\left( 1+\frac{1}{\delta }T_{n}\left( \left\vert t-x\right\vert
;x\right) \right) \omega \left( g;\delta \right) \\
&\leq &\left( 1+\frac{1}{\delta }\sqrt{\Delta _{2}}\right) \omega \left(
g;\delta \right) .
\end{eqnarray*}%
Then from Lemma \ref{Lemma 4}, one has%
\begin{equation}
\left\vert T_{n}\left( g;x\right) -g\left( x\right) \right\vert \leq \left(
1+\frac{1}{\delta }\sqrt{\frac{1}{n^{2}}x\left( 4x^{3}\alpha ^{2}+4\alpha
x+n\right) +\frac{2\mu x}{n}\frac{e_{\mu }(-nx)}{e_{\mu }(nx)}}\right)
\omega \left( g;\delta \right) ,  \label{15}
\end{equation}%
by choosing $\delta =\frac{1}{\sqrt{n}}$, we completes the proof.
\end{proof}

Let $C_{B}[0,\infty )$ denote the space of uniformly continuous and bounded
functions on $[0,\infty )$. Also%
\begin{equation}
C_{B}^{2}[0,\infty )=\{g\in C_{B}[0,\infty ):g^{\prime },g^{\prime \prime
}\in C_{B}[0,\infty )\}  \label{16}
\end{equation}%
with the norm 
\begin{equation*}
\left\Vert g\right\Vert _{C_{B}^{2}[0,\infty )}=\left\Vert g\right\Vert
_{C_{B}[0,\infty )}+\left\Vert g^{\prime }\right\Vert _{C_{B}[0,\infty
)}+\left\Vert g^{\prime \prime }\right\Vert _{C_{B}[0,\infty )}
\end{equation*}
for $\forall g\in C_{B}^{2}[0,\infty ).$

\begin{lemma}
\label{Lemma 8}For $h\in C_{B}^{2}[0,\infty )$, the following inequality
holds true%
\begin{equation}
\left\vert T_{n}\left( h;x\right) -h\left( x\right) \right\vert \leq \left[
\Delta _{1}+\Delta _{2}\right] \left\Vert h\right\Vert _{C_{B}^{2}[0,\infty
)},  \label{17}
\end{equation}%
where $\Delta _{1}$ and $\Delta _{2}$ are given by in Lemma \ref{Lemma 4}.
\end{lemma}

\begin{proof}
From the Taylor's series of the function $h$,%
\begin{equation*}
h\left( s\right) =h\left( x\right) +\left( s-x\right) h^{\prime }\left(
x\right) +\frac{\left( s-x\right) ^{2}}{2!}h^{\prime \prime }\left( \varrho
\right) ,\text{ }\varrho \in \left( x,s\right) .
\end{equation*}%
Applying the operator $T_{n}$ to both sides of this equality and then using
the linearity of the operator, we have%
\begin{equation*}
T_{n}\left( h;x\right) -h\left( x\right) =h^{\prime }\left( x\right) \Delta
_{1}+\frac{h^{\prime \prime }\left( \varrho \right) }{2}\Delta _{2}.
\end{equation*}%
From Lemma \ref{Lemma 4}, it yields%
\begin{eqnarray*}
\left\vert T_{n}\left( h;x\right) -h\left( x\right) \right\vert &\leq &\frac{%
2\alpha x^{2}}{n}\left\Vert h^{\prime }\right\Vert _{C_{B}[0,\infty )} \\
&&+\left[ \frac{1}{n^{2}}x\left( 4x^{3}\alpha ^{2}+4\alpha x+n\right) +\frac{%
2\mu x}{n}\frac{e_{\mu }(-nx)}{e_{\mu }(nx)}\right] \left\Vert h^{\prime
\prime }\right\Vert _{C_{B}[0,\infty )} \\
&\leq &[\Delta _{1}+\Delta _{2}]\left\Vert h\right\Vert _{C_{B}^{2}[0,\infty
)},
\end{eqnarray*}%
which finishes the proof.
\end{proof}

Now we recall that the second order of modulus continuity of $f$ on $%
C_{B}[0,\infty )$ is given as%
\begin{equation*}
\omega _{2}\left( f;\delta \right) :=\sup_{0<s\leq \delta }\left\Vert
f\left( .+2s\right) -2f\left( .+s\right) +f\left( .\right) \right\Vert
_{C_{B}[0,\infty )}.
\end{equation*}%
Peetre's $K$-functional of the function $f\in C_{B}\left[ 0,\infty \right) $
is as follows%
\begin{equation}
K\left( f;\delta \right) :=\inf_{g\in C_{B}^{2}\left[ 0,\infty \right)
}\left\{ \left\Vert f-g\right\Vert _{C_{B}}+\delta \left\Vert g\right\Vert
_{C_{B}^{2}}\right\}  \label{a13}
\end{equation}%
The relation between $K$ and $\omega _{2}$ is as%
\begin{equation}
K\left( f;\delta \right) \leq M\left\{ w_{2}\left( f;\sqrt{\delta }\right)
+\min \left( 1,\delta \right) \left\Vert f\right\Vert _{C_{B}}\right\}
\label{a11}
\end{equation}%
for all $\delta >0.$ Here $M$ is a positive constant. Now, we can give the
important theorem.

\begin{theorem}
\label{Theorem 9}For the operators by (\ref{1001}), the following inequality
holds%
\begin{equation}
\left \vert T_{n}\left( g;x\right) -g\left( x\right) \right \vert \leq
2M\left \{ \min \left( 1,\frac{\chi _{n}\left( x\right) }{2}\right) \left
\Vert g\right \Vert _{C_{B}[0,\infty )}+\omega _{2}\left( g;\sqrt{\frac{\chi
_{n}\left( x\right) }{2}}\right) \right \}  \label{20}
\end{equation}%
where $\forall g\in C_{B}[0,\infty ),\ x\in \lbrack 0,\infty )$, $M$ is a
positive constant which is independent of $n$ and $\chi _{n}\left( x\right)
=\Delta _{1}+\Delta _{2}$.
\end{theorem}

\begin{proof}
For any $f\in C_{B}^{2}[0,\infty )$, from the triangle inequality, we can
write%
\begin{equation*}
\Theta =\left\vert T_{n}\left( g;x\right) -g\left( x\right) \right\vert \leq
\left\vert T_{n}\left( g-f;x\right) \right\vert +\left\vert T_{n}\left(
f;x\right) -f\left( x\right) \right\vert +\left\vert g\left( x\right)
-f\left( x\right) \right\vert
\end{equation*}%
from Lemma \ref{Lemma 8}, which follows%
\begin{eqnarray*}
\Theta &\leq &2\left\Vert g-f\right\Vert _{C_{B}[0,\infty )}+\chi _{n}\left(
x\right) \left\Vert f\right\Vert _{C_{B}^{2}[0,\infty )} \\
&=&2\left\{ \left\Vert g-f\right\Vert _{C_{B}[0,\infty )}+\frac{\chi _{n}}{2}%
\left( x\right) \left\Vert f\right\Vert _{C_{B}^{2}[0,\infty )}\right\} .
\end{eqnarray*}%
From (\ref{a13}), we have%
\begin{equation*}
\Theta \leq 2K\left( g;\frac{\chi _{n}\left( x\right) }{2}\right) ,
\end{equation*}%
which holds%
\begin{equation*}
\Theta \leq 2M\left\{ \min \left( 1,\frac{\chi _{n}\left( x\right) }{2}%
\right) \left\Vert g\right\Vert _{C_{B}[0,\infty )}+\omega _{2}\left( g;%
\sqrt{\frac{\chi _{n}\left( x\right) }{2}}\right) \right\}
\end{equation*}%
from (\ref{a11}).
\end{proof}

Similar to the proof of above theorem, simple computations give the next
theorem.

\begin{theorem}
\label{Theorem10}If $g\in C_{B}[0,\infty )$ and $x\in \lbrack 0,\infty )$,
we get 
\begin{equation*}
\left\vert T_{n}\left( g;x\right) -g\left( x\right) \right\vert \leq M\omega
_{2}\left( g;\frac{1}{2}\sqrt{\frac{1}{n^{2}}x\left( 8x^{3}\alpha
^{2}+4x\alpha +n\right) +\frac{2\mu x}{n}\frac{e_{\mu }(-nx)}{e_{\mu }(nx)}}%
\right) +\omega \left( g;\frac{2\alpha x^{2}}{n}\right) 
\end{equation*}%
where $M$ is a positive constant.
\end{theorem}

\begin{remark}
The case of $\mu =0$ in Theorem \ref{Theorem10} gives the result given in 
\cite{K}.
\end{remark}

\end{document}